\documentclass[11pt]{amsart}
\usepackage[title]{appendix}
\usepackage[a4paper,top=3cm,bottom=2cm,left=3cm,right=3cm,marginparwidth=1.75cm]{geometry}
\usepackage[noabbrev,capitalise]{cleveref}
\usepackage{amsmath,amssymb,amsthm,amsfonts,mathrsfs,mathtools,relsize}
\usepackage{xcolor}
\usepackage{thm-restate}

\title{Typical structure of hereditary properties of binary matroids}
\author{Stefan Grosser}
\address{Department of Mathematics and Statistics, McGill University, Montreal, QC.}
\email{stefan.grosser@mail.mcgill.ca}
\author{Hamed Hatami}
\address{School of Computer Science, McGill University, Montreal, QC.}
\email{hatami@cs.mcgill.ca}
\author{Peter Nelson}\address{Department of Combinatorics and Optimization, University of Waterloo, Waterloo, Canada}
\email{apnelson@uwaterloo.ca}
\author{Sergey Norin}
\address{Department of Mathematics and Statistics, McGill University, Montreal, QC.}
\email{snorin@math.mcgill.ca}

\thanks{HH, PN and SN were supported by NSERC Discovery Grants. PN was additionally supported by an Early Researcher Award from the government of Ontario. SG received support from Fonds de Recherche du Qu\'ebec Nature et Technologies.}

\newtheorem{thm}{Theorem}[section]

\newtheorem{definition}[thm]{Definition}
\newtheorem{lem}[thm]{Lemma}
\newtheorem{cor}[thm]{Corollary}

\newtheorem{proposition}[thm]{Proposition}

\theoremstyle{remark}

\newcommand{\cP}{\mathcal{P}}

\newcommand{\e}{\mathbb{E}}

\usepackage[shortlabels]{enumitem}

\newcommand{\mc}[1]{\mathcal{#1}}

\newcommand{\bb}[1]{\mathbb{#1}}

\newcommand{\brm}[1]{\operatorname{#1}}

\newcommand{\eps}{\varepsilon}

\newcommand{\s}[1]{\left(#1\right)}

\newcommand{\Sub}{\brm{Sub}}

\begin{document}

\begin{abstract}
We prove an arithmetic analogue of the typical structure theorem for graph hereditary properties due to Alon, Balogh, Bollob\'as and Morris~\cite{ABBM11}.
\end{abstract}
\maketitle

\section{Introduction}

Recently developed theory of higher order Fourier analysis (see e.g. monographs by Tao~\cite{Tao12} and Hatami, Hatami and Lovett~\cite{HHL19}) provides a robust framework for establishing arithmetic analogues of results in extremal graph theory. In this framework, decomposition theorems for functions over $\bb{F}^n_p$~\cite{BFHHL13,Gowers10,Green05} play a role analogous to Szemeredi's regularity lemma~\cite{Szemeredi75}. 

The parallels are particularly natural in the setting of binary matroids. A \emph{simple binary matroid} (henceforth a \emph{matroid}) is a function $M : V(M) \to \{0,1\}$ where $V(M)$ is a binary projective space, i.e. $V(M)  \simeq PG(n-1,2) = \bb{F}_2^n \setminus \{0\}$ for some $n \in \bb{N}$.\footnote{In the comparisons we make with results in graph theory, the space $V$ is analogous to the vertex set of a graph, and the function $M$ is analogous to the indicator function of the edge set.} The above mentioned decomposition theorems were used  in this setting, in particular, to establish an analogue of the classical Erd\H{o}s-Stone theorem and a corresponding stability theorem~\cite{LLNN20}, and to explore the structure of dense matroids in monotone properties~\cite{Campbell16,GeeNel16,Luo19}. 

We focus our attention on hereditary matroid properties.
A \emph{matroid property} is a set of matroids closed under isomorphism, where a matroid isomorphism between $M_1$ and $M_2$ is an invertible linear map $\phi: \bb{F}_2^n  \to \bb{F}_2^n$  such that  $M_1 = M_2 \circ \phi $. A matroid property is \emph{hereditary} if it is further closed under restriction to linear subspaces. Meanwhile, a \emph{graph property} is a set of graphs closed under isomorphism. A graph property is \emph{hereditary} if it is closed under taking induced subgraphs. 

One  striking analogy between hereditary properties of graphs and matroids comes from the area of property testing. A seminal result of Alon and Shapira~\cite{AloSha08} establishes that a graph property has an oblivious one-sided error tester if and only if it is (semi)-hereditary. Bhattacharyya, Grigorescu and Shapira~\cite{BGS15} conjectured the exact analogue of this result for binary matroids.\footnote{A different vocabulary is used in the property testing, and hereditary matroid properties are referred to as \emph{linear-invariant subspace hereditary properties of boolean functions.}} Following a series of partial results~\cite{BGS15,BFL12,BFHHL13}, this  conjecture has recently been established by Tidor and Zhao~\cite{TidZha19}.

We pursue the analogy between hereditary properties of graphs and  binary matroids
from a different angle, focusing on the typical structure.
The study of typical structure of graphs in a given hereditary property was initiated by Erd\"os, Kleitman and Rothschild~\cite{EKR76}, and it has been extensively investigated since. Initial focus was primarily on the typical structure of \emph{principal} hereditary properties consisting of all $H$-free graphs for a fixed graph $H$, where a graph $G$ is \emph{$H$-free} if $G$ does not contain an induced subgraph isomorphic to $H$. For example,  Kolaitis, Pr\"{o}mel and Rothschild~\cite{KPR87} proved that almost all $K_{k+1}$-free graphs are $k$-colorable, extending the result of \cite{EKR76} for $k=2$. Pr\"{o}mel and Steger~\cite{PS91} proved that vertices of almost all $C_4$-free graphs can be partitioned into a clique and a stable set. 

A particularly general theorem in this direction has been established by Alon, Balogh, Bollob\'as and Morris~\cite{ABBM11}. Our matroid results are modeled on their theorem, which we now introduce in detail.

\subsection*{Alon-Balogh-Bollob\'as-Morris theorem }
Informally, the main result of~\cite{ABBM11}  says that for every hereditary graph property $\mc{P}$, there exists a constant $k$ such that for almost every graph in $G \in \mc{P}$, the vertices of $G$ can be divided into $k$ parts with negligible leftover
so that the subgraphs induced by the parts are structured, whereas the edges between the parts are essentially arbitrary. 

We now formalize the above. Let $[n]=\{1,2,\ldots,n\}$.
For a  graph property $\mc{P}$, let $\mc{P}^n$ denote the set of graphs $G \in \mc{P}$ with $|V(G)|=[n]$.
We say that property $\mc{P}_*$ holds for \emph{almost all graphs} in $\mc{P}$, if $$\lim_{n \to \infty}\frac{|\mc{P}^n \cap \mc{P}_{*}|}{|\mc{P}^n|} = 1.$$ 
 
The constant $
k$ which appears in the above informal description is the coloring number of $\mc{P}$, defined as follows. Given integers $s,t \geq 0$, let $\mc{H}(s,t)$ denote the family of all graphs $G$ such that $V(G)$ admits a partition into $s$ cliques and $t$ independent sets. In particular, $\mc{H}(s,0)$ is the family of all $s$-colorable graphs. A graph property is  \emph{non-trivial} if it contains arbitrarily large graphs. The \emph{coloring number} $\chi_c(\mc{P})$ of a non-trivial hereditary graph property $\mc{P}$ is the maximum integer $k$ such that $\mc{H}(s,k-s) \subseteq \mc{F}$ for some $0 \leq s \leq k$.

The \emph{entropy of $\mc{P}$}, which naturally measures the size of  $\mc{P}$,  is the sequence $h(n,\mc{P}) = \log|\mc{P}^n|$, where here and throughout the paper all logarithms are base $2$.   Alekseev~\cite{Alekseev92} and, independently, Bollob\'{a}s and Thomason~\cite{BolTho97}, gave an asymptotic expression for the entropy of hereditary properties in terms of their coloring number, showing that
\begin{equation}\label{e:Gentropy}
h(n,\mc{P})  =  \left(1 - \frac{1}{\chi_c(\mc{P})} + o(1)\right)\frac{n^2}{2}.
\end{equation}
The above expression gives the correct asymptotic order of the entropy, unless $\chi_c(\mc{P}) \leq 1$, in which case   it only tells us that $\mc{P}$ is subquadratic.  

We say $\mc{P}$ is \emph{thin} if  $\chi_c(\mc{P}) \leq 1$. Note that, by definition of $\chi_c(\mc{P})$, a hereditary graph property $\mc{P}$ is thin if and only if $\mc{P}$ does not contain all the bipartite graphs, nor all the complements of bipartite graphs, nor all the split graphs, where a graph is \emph{split} if its vertex set can be partitioned into a clique and a stable set. Thin properties have much lower entropy than non-thin ones, and thus  we  shall think of the graphs in a fixed thin property to be structured. For example,   in   our informal description of Alon-Balogh-Bollob\'as-Morris's theorem, it is in this sense that  the subgraphs induced by each of the $k$ parts are structured.

Alon et al. give an equivalent description of thin graph properties in terms of forbidden bigraphs. A \emph{bigraph} is a triple $H_{A,B}=(H,A,B)$, where $H$ is a bipartite graph and $(A,B)$ is a bipartition of $H$.  
An injection $\phi: V(H) \to V(G)$  is \emph{an  $H_{A,B}$-instance  in a graph $G$} if for  all $u \in A$ and $v \in B$ we have $uv \in E(H)$ if and only if $\phi(u)\phi(v) \in E(G)$. We say that a graph $G$ is \emph{$H_{A,B}$-free} if there are no $H_{A,B}$-instances in $G$. We are now also ready  to formally state the main structural results of~\cite{ABBM11}. 

\begin{thm}[{\cite[Theorem 2, Lemma 7]{ABBM11}}]\label{t:Gthin} Let $\mc{T}$ be a non-trivial hereditary graph property.  The following are equivalent: 
	\begin{itemize}
		\item $\mc{T}$ is thin (i.e. $\chi_c(\mc{P}) \leq 1$),
		\item there exists a bigraph $H_{A,B}$ such that every graph in $\mc{T}$ is $H_{A,B}$-free,
		\item there exists $\eps > 0$ such that $ h(n,\mc{T}) \leq n^{2-\eps}$ for all $n$.
	\end{itemize}	
\end{thm}	

\begin{thm}[{\cite[Theorem 1]{ABBM11}}]\label{t:Gmain}
	Let $\cP$ be a non-trivial hereditary graph property with coloring number $\chi_c(\cP) = k$. There exists an $\eps> 0$ and a thin hereditary property $\mc{T}$ such that	
	for almost all graphs $G \in \cP$, there exists a partition $(S_1, \ldots, S_k, Z)$ of $V(G)$ satisfying
	\begin{itemize}
		\item $G[S_i] \in \mc{T}$ for every $1 \leq i \leq k$, and
		\item $|Z| \leq n^{1 - \eps}$.
	\end{itemize}
\end{thm}

It is not hard to see that Theorems~\ref{t:Gthin} and ~\ref{t:Gmain} imply that \eqref{e:Gentropy} can be strengthened by replacing the $o(1)$ term by $o(n^{-\eps})$ for some $\eps>0$ depending only on $\mc{P}$.

\subsection*{Our results}

We  translate the concepts appearing in the statements of Theorems~\ref{t:Gthin} and ~\ref{t:Gmain} to our setting as follows. 

We define the \emph{dimension $\dim(V)$} of a binary projective space $V$ to be equal to the dimension of the $\bb{F}_2$-vector space $W$ such that $V=W \setminus  \{0\}$.  We caution the reader that  our definition differs  (by one)   from the standard definition of dimension of projective spaces, but appears to be more natural in the binary setting. The \emph{dimension} $\dim(M)$ of a matroid $M$ is the dimension of $V(M)$. That is if $M:\bb{F}_2^n \setminus\{0\} \to \{0,1\}$, then $\dim(M)=n$. 

We say that a matroid property is \emph{non-trivial} if it contains matroids of arbitrarily high dimension. Let $\mc{O}$ and $\mc{I}$ denote the hereditary properties consisting of all identically zero matroids and of all identically one matroids, respectively. The matroid Ramsey theorem (\cref{t:MRamsey} below) implies that $\mc{O}$ and  $\mc{I}$ are the only minimal non-trivial hereditary matroid properties.

For  an integer $k \geq 0$,
let $\brm{Ext}^k(M)$ denote the set of all extensions of $M$ to  matroids of dimension at most $k+\dim(M)$. In other words, $M' \in  \brm{Ext}^k(M)$ if and only if $M$ is a restriction of $M'$ to a  subspace of  codimension at most $k$ of $V(M')$. For a property $\mc{P}$, let 
$$\brm{Ext}^k(\mc{P})=\bigcup_{ M \in \mc{P}} \brm{Ext}^k (M).$$
 Note that if $\mc{P}$ is hereditary, then so is $\brm{Ext}^k(\mc{P})$.

For a matroid $M$ and a subspace $W$ of $V(M)$ we denote by $M[W]$ the restriction of $M$ to $W$. Let $\mc{M}(k,0) = \brm{Ext}^k(\mc{O})$, and note that  $\mc{M}(k,0)$ is the family of all  matroids $M$ such that $M[W] \equiv 0$ for some subspace $W$ of $V(M)$ of codimension at most $k$.

\begin{definition}[Critical number]
The critical number of a matroid $M$ is the smallest codimension of a linear subspace $W \subseteq V(M)$ such that $M[W] \equiv 0$. 
\end{definition}

Hence $\mc{M}(k,0)$ is the set of all matroids with  critical number at most $k$. 
It has been noted that  families of matroids with critical number $k$ are  in many respects  similar to families of graphs with chromatic number $k+1$.  (See e.g. ~\cite{GeeNel16}.) Symmetrically, let $\mc{M}(k,1) = \brm{Ext}^k(\mc{I})$.

%It has been noted that  families of matroids with bounded critical number are  in many respects  similar to families of graphs with bounded chromatic number \textcolor{red}{Perhaps this is worth a citation?}.

The families $\mc{M}(k,0)$ and $\mc{M}(k,1)$ are the matroid equivalents of graph families $\mc{H}(s,t)$ with $s+t=k+1$.

\begin{definition} 
The \emph{critical  number} $\chi_c(\mc{P})$ of a non-trivial hereditary matroid property $\mc{P}$ is the maximum integer $k \geq 0$ such that $\mc{M}(k,i) \subseteq \mc{P}$ for some $i \in \{0,1\}$.  
\end{definition}

We say that a non-trivial  hereditary  matroid property is \emph{thin} if $\chi_c(\mc{P}) = 0$. 
Note that $\chi_c(\mc{P}) < k$ if there exist  matroids $M_0,M_1 \not \in \mc{P}$    such that  for  $i \in \{0,1\}$   we have  $M_i[W_i] \equiv i$ for some  linear subspace $W_i \subseteq V(M_i)$ of codimension $k$.

Since by the matroid Ramsey theorem  $\mc{O}$ and  $\mc{I}$ are the only minimal non-trivial hereditary matroid properties, $\chi_c(\mc{P})$ can be equivalently defined as the maximum integer $k$ such that $\brm{Ext}^k(\mc{T}) \subseteq \mc{P}$  for some  non-trivial hereditary matroid property $\mc{T}$.

A linear injection $\phi: V(N) \to V(M)$  is \emph{an  $N$-instance  in a matroid $M$} if $N(x)=M(\phi(x))$ for every $x \in V(N)$. Often, we only need  the equality $N(x)=M(\phi(x))$ to hold only for $x$ in some subset of $V(N)$, just as for the bigraph instances we did not insist that the parts of the bipartition are mapped onto independent sets. This motivates the following key definition. A \emph{pattern} is a function $N : V(N) \to \{0,1,\star\}$ where $V(N)$  is a binary projective space. For patterns $N$ and $M$, an {$N$-instance in  $M$} is a  linear injection $\phi: V(N) \to V(M)$ such that $N(x)=M(\phi(x))$ for every $x \in V(N)$ such that $N(x) \in \{0,1\}$.   We say that a matroid $M$ is \emph{$N$-free} for a pattern $N$ if  there are no $N$-instances in $M$.

\begin{definition}
	For an integer $k \ge 1$, a pattern $A$ is said to be \emph{$k$-affine} if  $ A^{-1}(\star)$ is a subspace of codimension $k$ in $V(A)$. 
\end{definition}

A important example of a $k$-affine pattern is the pattern corresponding to {\em Bose-Burton geometry}.

\begin{definition}[Bose-Burton pattern]
	\label{def:BB}
	The {\em Bose-Burton pattern} $B=BB_{k,d}$ is the $k$-affine pattern with $\dim(B)=d$ such that for some subspace $W \subseteq V(B)$ of codimension $k$,  we have  $B\vert_{W} \equiv \star$, and $B\vert_{V \setminus W} \equiv 1$.
\end{definition}

Note that setting the $\star$'s of a $k$-affine pattern to $0$ yields a matroid with critical number at most $k$. In particular, $1$-affine patterns are matroidal analogues of bigraphs. We simply refer to   $1$-affine patterns  as   \emph{affine patterns}.

%Let $\mc{P}^n$ denote the set of all matroids $M \in \mc{P}$ with $V(M)=PG(n-1,2)$. 

Let $\mc{P}^n$ denote the set of all matroids $M \in \mc{P}$ with $\dim(M)=n$.  As for graphs, the \emph{entropy} of $\mc{P}$ is defined as  $h(n,\mc{P}) = \log|\mc{P}^n|.$
We are now ready to state our analogues of Theorems~\ref{t:Gthin} and~\ref{t:Gmain}. 

\begin{restatable}{thm}{Thin}\label{t:thin} Let $\mc{T}$ be a non-trivial hereditary property of binary matroids. Then the following are equivalent:  \begin{itemize}
		\item[(a)] $\mc{T}$ is thin (i.e. $\chi_c(\mc{T}) = 0$),   
		\item[(b)] there exists an affine pattern $A$ such that every matroid in $\mc{P}$ is $A$-free,
		\item [(c)] $h(n,\mc{T})=o(2^n)$.
	\end{itemize}
\end{restatable}

\begin{restatable}{thm}{Main}\label{t:main}
	Let $\cP$ be a non-trivial hereditary property of binary matroids  with $\chi_c(\cP) = k$. Then there exists a thin hereditary property $\mc{T}$ such that almost every matroid in $\mc{P}$ belongs to $ \brm{Ext}^k(\mc{T})$.	
\end{restatable}

To parallel  our informal introduction of Theorems~\ref{t:Gthin} and \ref{t:Gmain}, we can informally describe Theorems~\ref{t:thin} and \ref{t:main}  as stating that for every hereditary matroid property $\mc{P}$ there exists a constant $k$ such that almost every matroid in $\mc{P}$  is structured on a codimension $k$ subspace and is essentially arbitrary outside of it.

As $$h(n-k,\mc{T}) +\s{1-\frac{1}{2^{k}}}2^n \leq  h(n,\brm{Ext}^k(\mc{T})) \leq h(n-k,\mc{T})  +\s{1-\frac{1}{2^{k}}}2^n + kn,$$
Theorems~\ref{t:thin} and \ref{t:main} immediately imply an asymptotic formula for the entropy of every hereditary matroid property, which depends only on its critical number.

\begin{cor} Let $\cP$ be a non-trivial  hereditary matroid property then 
 	$$ h(n,\mc{P})= \s{1 - \frac{1}{2^{\chi_c(\cP)}} + o(1)}2^n.$$
\end{cor}

Let us point out a couple of places where Theorems~\ref{t:thin}  and \ref{t:main} differ from their graph theoretical analogues. 

\begin{itemize}
	\item The bound on entropy of thin properties is weaker. We do not know if a bound of the form $h(n,\mc{T})=o(2^{(1-\eps)n})$ holds for every thin hereditary property $\mc{T}$.
	\item On the other hand, the structure given by Theorem~\ref{t:main} is cleaner, requiring no small leftover set $Z$.
\end{itemize}

Theorems~\ref{t:Gmain} and~\ref{t:main} only give a ``rough'' structural description for typical elements of hereditary families $\mc{P}$, in a sense that a matroid (or a graph) that matches the structural description does not necessarily belong to $\mc{P}$. \footnote{Although, for matroids, \cref{t:main} guarantees that such a matroid belongs to a fixed hereditary family with the same critical number and thus similar entropy.} Meanwhile, the typical structure of $K_{k+1}$-free and $C_4$-free graphs established described above guarantees the absence of $K_{k+1}$ and $C_4$, respectively. It appears desirable to improve Theorem~\ref{t:main} in the same way, i.e. to show that for a given hereditary matroid property $\mc{P}$ there exists a hereditary family $\mc{P}_* \subseteq \mc{P}$ with a structural description as in \cref{t:main}  such that almost every matroid in $\mc{P}$ lies in $\mc{P}_*$.  

We are able to do so for a natural class of locally characterized hereditary matroid properties, defined as follows. For a set $\mc{N}$ of matroids, let $\brm{Forb}(\mc{N})$ denote the set of matroids $M$ such that $M$ is $N$-free for every $N \in \mc{N}$. Then $\brm{Forb}(\mc{N})$ is a hereditary property for every $\mc{N}$. Conversely, for every hereditary matroid property $\mc{P}$ we have $\mc{P}=\brm{Forb}(\mc{N})$ for some $\mc{N}$. A hereditary matroid property $\mc{P}$ is \emph{locally characterized} if there exists a  finite  $\mc{N}$ such that $\mc{P}=\brm{Forb}(\mc{N})$. 

For an integer $k \geq 0$ and a matroid property $\mc{P}$  we define 
 $$\brm{Core}^k(\mc{P}) = \{M \: | \: \brm{Ext}^k(M) \subseteq \mc{P}\}.$$
Thus
  $$\brm{Ext}^k(\brm{Core}^k(\mc{P})) \subseteq \mc{P}.$$  
Note that if $\mc{P}$ is a non-trivial hereditary matroid property, then $\chi_c(\mc{P}) \ge k$ if and only if  $\brm{Core}^{k}(\mc{P})$ is non-trivial, and $\chi_c(\mc{P}) \le k$ if and only if $\brm{Core}^{k}(\mc{P})$ is thin. Hence  $\brm{Core}^{k}(\mc{P})$ is non-trivial and thin if and only if $k=\chi_c(\mc{P})$.
 The following is our last main result.

\begin{restatable}{thm}{Local}\label{t:local}
	Let $\cP$ be a locally characterized  hereditary matroid property with $\chi_c(\cP) = k$, and let $\mc{T} = \brm{Core}^k(\mc{P})$. Then almost every matroid in $\mc{P}$ belongs to $\brm{Ext}^k(\mc{T})$.
\end{restatable}

Once again let us informally rephrase \cref{t:local}. It states that for almost every $M \in \mc{P}$, where $\mc{P}$   is as in \cref{t:local}, there exists a codimension $k$ subspace $W$ of $V(M)$ such that
$M[W]$ is structured, while the structure of $M$ outside of $W$ does not matter at all, i.e.  every matroid obtained  by changing values of $M$ on $V(M) \setminus  W$ still lies in $\mc{P}$.

The natural analogue of \cref{t:local} is known not to hold for locally characterized graph hereditary properties, see e.g.~\cite[Theorem 3]{BBS11}  for an example when  a bounded number of exceptional vertices needs to be introduced in the description. The analogue for principal graph hereditary properties was conjectured by Reed and Scott, but was disproved in~\cite{NorYud192}.

\vskip 5pt
In many cases \cref{t:local} can be straightforwardly applied to explicitly describe typical structure of given hereditary families. Let us sketch the argument for one fairly broad class of examples. 

Let $O_2$ be the matroid such that $\dim(O_2)=2$ and $O_2 \equiv 0$. Then $O_2$-free matroids can be considered as analogues of  graphs with independence number at most two. 
As an immediate consequence of \cref{t:local} we obtain a concrete typical structure of $M$-free matroids for any $O_2$-free matroid $M$.   

\begin{cor}\label{c:o2}
	Let $M:\bb{F}_2^k \setminus \{0\} \to \{0,1\}$ be an $O_2$-free matroid with $k \geq 3$, and let $c=|M^{-1}(0)|$.   
	\begin{itemize} \item if $c=0$ then almost every $M$-free matroid belongs to $\mc{M}(k- 1,0)$, 
		\item if $c \geq 2$ then almost every $M$-free matroid belongs to $\mc{M}(k- 2,0)$, and
		\item if $c=1$  then almost  every $M$-free matroid can be obtained from a matroid in $\mc{M}(k- 2,0)$ by changing at most one value. 
		\end{itemize}
\end{cor}
\begin{proof}[Proof sketch.]
	Let $\mc{P}=\brm{Forb}(M)$. Let $d=1$ if $c=0$, and $d=2$, otherwise. Note that $M[W] \not \equiv 0$ for every subspace $W$ of $V(M)$ with $\dim(W)=d$. 
	
	Let $\mc{T}=\brm{Core}^{k-d}(\mc{P})$. Then every matroid in $\mc{T}$ is $M[W]$-free for every subspace $W$ of $V(M)$ with $\dim(W)=d$. In the case, $c \neq 1$, this implies that every matroid in $\mc{T}$ of dimension at least two is identically zero, and in the case $c=1$ a single non-zero value is allowed. The corollary thus follows from \cref{t:local}.   
\end{proof}

The rest of the paper is devoted to the proofs of Theorems~\ref{t:thin},~\ref{t:main} and~\ref{t:local}.  In Section \ref{sec:tools}, we introduce the necessary tools from higher order Fourier analysis and extremal theory of binary matroids. In Section \ref{sec:kaffine}, we establish equivalence of conditions (a) and (b) in \cref{t:thin} and give a similar characterization of  properties with higher critical number. The proofs are finished in \cref{sec:entropy}.

\section{Tools}\label{sec:tools}

The proof of \cref{t:Gmain} in \cite{ABBM11} uses the classical toolkit of extremal graph theory: the Erd\H{o}s-Stone theorem and its stability version, Ramsey theorem and Szemeredi's regularity lemma.  Analogues of all of these ingredients exist for matroids and form the foundation of our proof. We introduce them in this section.

We start with the  arithmetic Ramsey theorem, which immediately follows from Graham-Rothschild theorem~\cite{GraRot69}, as observed, in particular, by Green and Sanders~\cite{GreSan15}. See also~\cite[Theorem 19]{BGS15} for a self-contained proof.

\begin{thm}\label{t:MRamsey}
For every $d \geq 1$, there exists $n=n_{\ref{t:MRamsey}}(d)$ such that for every matroid $M$ with $\dim(M) \geq n$, there exists a restriction $N$ of $M$ such that $\dim(N)=d$ and either $N \equiv 0$ or $N \equiv 1$.	
\end{thm}

Our most sophisticated tool, the decomposition theorem (\cref{t:decomp} below), provides for every matroid $M$, a bounded size approximation of $M$. If $M$ is $N$-free for a pattern $N$ then the resulting approximation is only approximately $N$-free, that is, it has small $N$-density defined as follows.

Let $N$ be a pattern and let $M$ be a matroid. The \emph{$N$-density in $M$}, denoted by $t(N,M)$ is the probability that a uniformly random linear injection $\phi: V(N) \to V(M)$ is an $N$-instance. Additionally, the approximations mentioned above are not matroids, but more general functions $f:V \to [0,1]$. For a pattern $N$ with $V(N)=U$ we define the \emph{$N$-density}  in  such a function $f$ as
	$$t(N, f) = \mathlarger{\mathop{\e}}_{\phi} \left[\prod_{x \in N^{-1}(1)} f(\phi(v)) \prod_{x \in N^{-1}(0)} (1-f(\phi(v))) \right],$$
where the expectation is with respect to uniform distribution on linear injections  $\phi : U \to V$.

 Recall from \cref{def:BB} that $BB_{k,d}$ denotes the $k$-affine $d$-dimensional Bose-Burton pattern, which is the $d$-dimensional pattern that is $\star$ inside a subspace of codimension $k$, and $1$ outside.    Note that any   matroid $M$ that is zero on a subspace $W$ of codimension $k$ is a  $BB_{k+1,d}$-free matroid with $|M^{-1}(1)|= (1-2^{-k})|V(M)|$.   The following result of Liu et al.~\cite{LLNN20} is a matroidal analogue of the stability version of the Erd\H{o}s-Stone theorem. 

\begin{thm}[{\cite[Theorem 1.3]{LLNN20}}] \label{t:stability}
	For all $d \in \bb{N}$ and $\eps > 0$  there exist  $\delta=\delta_{\ref{t:stability}}(d,\eps), D=D_{\ref{t:stability}}(d,\eps) > 0$ satisfying the following. Let $k \in \bb{Z}_{+}$, $k< d$, and let $M$ be a $BB_{k+1,d}$-free matroid with $\dim(M) \geq D$.  If $|M^{-1}(1)| \geq (1 - 2^{-k} -\delta)|V(M)|$ then there exists a  subspace  $W  \subset V(M)$ of codimension $k$ such that $|M^{-1}(1) \cap W| \leq \eps|V(M)|$.
\end{thm}

We also need a removal lemma due to Luo~\cite{Luo19} which we only apply and state for the special case of Bose-Burton geometries.
 
\begin{thm}[{\cite[Theorem 4.1]{Luo19}}] \label{t:removal}
For all $d \in \bb{N}$ and $\eps > 0$   there exist  $\delta=\delta_{\ref{t:removal}}(d,\eps) > 0$ satisfying the following. Let $M$ be a matroid such that $t(BB_{k+1,d},M) \leq \delta$ for some integer $0 \leq k \leq d-1$. Then there exists a  $BB_{k+1,d}$-free matroid $M_0$ with $V(M_0)=V(M)$ and $M^{-1}_0(1) \subseteq M^{-1}(1) $ such that $|M^{-1}(1) - M^{-1}_0(1)| \leq \eps |V(M)|$.
\end{thm}

Our final tool is a decomposition theorem for bounded functions over $\bb{F}^n_p$, which we use in the case $p=2$. This result requires the most introduction. Although this theorem is essential in our argument we need a surprisingly weak version of it, and so we only introduce the concepts necessary to state this weakening.

Let $\mathbb{T}=\mathbb{R} / \mathbb{Z}$. For $f \colon \mathbb{F}_2^n \to \mathbb{T}$ and $y \in  \mathbb{F}_2^n$, let $(D_y f)(x)\coloneqq f(x+y) - f(x)$ be the {\em  derivative} of $f$ in the direction $y$. Let $d \geq 0$ be an integer. A {\em (nonclassical) polynomial}  of degree at most $d$ is a function $P\colon \mathbb{F}_2^n \to \mathbb{T}$ such that for all $y_1,y_2,\dots ,y_{d+1}, x \in \mathbb{F}_2^n$, we have \[(D_{y_1}D_{y_2}\cdots D_{y_{d+1}}P)(x) = 0.\]
  
A partition $\mathcal{B}$ of $\mathbb{F}_2^n$  is {\em a polynomial factor over $\mathbb{F}_2^n$  of complexity $C$ and degree $d$} if there exists polynomials  $P_1, \dots,P_C:\mathbb{F}_2^n \to \bb{T}$  of degree at most $d$ such that $x,y \in \mathbb{F}_2^n$  belong to the same part of $\mc{B}$ if and only if $P_i(x)=P_i(y)$ for all $i \in [C]$. 

Let $X$ be a finite set, let  $g: X \to  \bb{R}$ and let $\mc{B}$ be a partition of $X$. The \emph{expectation $\e[g|\mc{B}]: X \to  \bb{R}$ of $g$ with respect to $\mc{B}$} is constant  on each part of $B \in \mc{B}$ and equal to the average of $g$ over $B$. 

Finally, the statement of  the decomposition theorem involves a family of \emph{Gowers norms} $\| \cdot \|_{U_d}$ of  functions $\bb{F}^n_2 \to \bb{R}$. We do not need the definition of these norms, but only the following property which is a consequence of ~\cite[Theorem 2.3]{GowWol10} (see also ~\cite[Corollary 11.9]{HHL19} for a strengthening.)

\begin{lem}\label{l:ud} For every pattern $A$, there exists an integer $d = d_{\ref{l:ud}}(A) \geq 0$ such that for every $\eps > 0$, there exists $\delta=\delta_{\ref{l:ud}}(\eps,A) >0$ satisfying the following. Let $V = \bb{F}^n_2  \{0\}$ be a projective binary space. If   $f, g: \bb{F}^n_2 \to [0,1]$ are such that $\|f-g\|_{U_{d+1}} \leq \delta$, then $|t(A,f|_V)-t(A,g|_V)| \leq \eps $.
\end{lem}

We are now ready to state the weak form of the decomposition theorem from \cite{HHL19} that we use in this paper. 

\begin{thm}[{\cite[Theorem 9.1]{HHL19}}]\label{t:decomp}
	For all $d \in \bb{N}, \delta> 0$ there exists $C=C_{\ref{t:decomp}}(d, \delta) > 0$ satisfying the following. For every $n \in \bb{N}$ and every $g \colon \mathbb{F}_2^n \to [0,1]$ there exists a polynomial factor $\mc{B}$ over $\mathbb{F}_2^n$  of complexity $C$ and degree $d$ such that
	$$\|\: g-\e[g|\mc{B}]\:\|_{U_{d+1}} \leq \delta.$$
\end{thm}

Given a function  $f: X \to  \bb{R}$ we say that $g$ is \emph{$f$-structured}  if for every $a \in \brm{Im}(f)$ we have $$\e[g(x)|f(x)=a] = a.$$ That is, $f=\e[g|\mc{B}]$ where $\mc{B}=\{f^{-1}(a)\}_{a \in \brm{Im}(f)}$. Let $\mc{M}(f)$ denote the set of all $f$-structured functions $M: X \to \{0,1\}$.

We finish this section by deriving from \cref{t:decomp} a lemma which is the cornerstone of our approach. It states that for every pattern $A$ and every binary projective space $V$, there exists a small collection of functions, which are close to being $A$-free, such that every $A$-free matroid $M: V \to \{0,1\}$ is $f$-structured for some $f$ in this collection.
  
First, we need a bound on the number of polynomial factors of given degree and complexity, which follows from the following characterization of polynomials of given degree, due to Tao and Ziegler~\cite{TaoZie12}.

\begin{lem}[{\cite[Lemma 1.6 (iii)]{TaoZie12}}]\label{l:poly}
	Let $P\colon \mathbb{F}_2^n \to \mathbb{T}$ be a polynomial of degree at most $d$. Then $$P(x_1,\ldots,x_n) = \alpha + \sum_{\substack{I \subseteq [n], j \in \bb{Z}_+ \\ |I|+j \leq d}} \frac{c_{I,j}}{2^j}\prod_{i \in I}|x_i| \qquad (\brm{mod} 1),$$
	for some  choice of coefficients $\alpha \in \bb{T}$ and $c_{I,j} \in \{0,1\}$,
	where $x \to |x|$ is the natural map $\bb{F}_2 \to \{0,1\}$.
\end{lem}

\begin{cor}\label{c:numPoly}
	Let $d,C,n$ be positive integers. Then there are at most $n^{dC}$ distinct polynomial factors over $\mathbb{F}_2^n$  of complexity $C$ and degree  $d$. Moreover, $|\mc{B}| \leq 2^{dC}$ for each such factor $\mc{B}$. 
\end{cor}

\begin{proof}
	By \cref{l:poly} each polynomial $P\colon \mathbb{F}_2^n \to \mathbb{T}$ of degree at most $d$ is determined by the choice of at most $d\binom{n}{d}  \leq n^d$ coefficients $c_{I,j}$, up to a constant shift.
	
	As every polynomial factor  of complexity $C$ and degree $d$ is determined by a $C$-tuple of such polynomials, which we may assume to be homogeneous without loss of generality, there are at most  $(n^{d})^{C}$ distinct polynomial factors over $\mathbb{F}_2^n$  of complexity $C$ and degree  $d$, as required.
	
	Moreover, \cref{l:poly} implies that each polynomial of degree at most $d$ takes at most $2^d$ distinct values. As every part of a polynomial factor $\mc{B}$ is determined by a $C$-tuple   of values of such polynomials, we also get the claimed upper bound on $|\mc{B}|$. 
\end{proof}

The next lemma shows that for every pattern $A$, the set of all $A$-free matroids can be covered with a union of a few  $\mc{M}(f_i)$, where each $f_i$ is almost $A$-free.  
\begin{lem}\label{t:bd}
	For every pattern  $A$ and $\delta > 0$ there exists $C=C_{\ref{t:bd}}(A,\delta)>0$ satisfying the following. For every  binary projective space $V$ there exist functions $f_1,f_2,\ldots,f_m: V \to [0,1]$ such that\begin{itemize}
		\item $m \leq 2^{C \cdot \dim(V)}$,
		\item $t(A,f_i) \leq \delta$ for every $i \in [m]$,
		\item for every $A$-free $M : V \to \{0,1\}$ there exists $i \in [m]$ such that $M \in \mc{M}(f_i).$ 
	\end{itemize} 	 
\end{lem}

\begin{proof}Let $d = d_{\ref{l:ud}}(A)$,   $\delta'= \delta_{\ref{l:ud}}(A, \delta/2)$, $C' =C_{\ref{t:decomp}}(d, \delta')$ and $$C=\max\left\{\frac{2|V(A)|}{\delta},\; dC'+2^{dC'}\right\}.$$
	
	Let $n = \dim(V)$.  If $|V| \leq \frac{2|V(A)|}{\delta}$ then $2^{C \cdot \dim(V)} \geq 2^{|V|}$. In this case the set $f_1,f_2,\ldots,f_m$ of all $A$-free matroids $M : V \to \{0,1\}$ satisfies the lemma. Thus we assume that \begin{equation}\label{e:bdv}|V| \geq \frac{2|V(A)|}{\delta}.\end{equation}
	 We identify $V$ with $\bb{F}^n_2 \setminus \{0\}$. Let $\mc{F}$ be the set of functions $f: V \to [0,1]$ such that there exists a polynomial factor $\mc{B}$ of degree $d$ and complexity $C'$ over $\bb{F}^n_2$ and a matroid $M: V \to \{0,1\}$ such that $f=\e[M|\mc{B}|_V]$.
	
	In particular, $f$ is completely determined by $\mc{B}$ and $\{ \sum_{x \in B}M(x) \}_{B \in \mc{B}|_V}$. This observation together with  \cref{c:numPoly} implies that
		$$|\mc{F}| \leq n^{dC'} \cdot (2^n)^{2^{dC'}} \leq 2^{(dC' + 2^{dC'})n} = 2^{Cn}.$$
			
	Thus it remains to show that for every $A$-free $M : V \to \{0,1\}$  there  exists $f \in \mc{F}$ such that $M$ is $f$-measurabl and $t(A,f|_V) \leq \delta$. This essentially follows from our choice of $C'$ to satisfy the conditions of \cref{t:decomp}, except for a minor technical difficulty: \cref{t:decomp} applies to functions $\bb{F}^n_2 \to [0,1]$ and $V$ differs from  $\bb{F}^n_2$ by one element. 
	
	We resolve this issue as follows. Let $M_0: \bb{F}^n_2 \to \{0,1\}$ be such that $M_0(0)=0$ and $M_0|_V = M$. By the choice of $C'$ there exists a polynomial factor $\mc{B}$ of complexity $C'$ and degree $d$ over $\bb{F}^n_2$ such that $\|M_0 - f_0\|_{U_{d+1}} \leq \delta'$, where $f_0=\e[M_0 |\mc{B}]$.
	 
	As $t(A,M)=0$ by the choice of $d$ and $\delta'$ it follows that 
	$t(A,f_0|_V) \leq \delta/2.$
	Let $f= \e[M |\mc{B}|_V]$ then   $f \in \mc{F}$, $M$ is $f$-structured, and $\|f-f_0|_V \|_1 \leq 1$. 
	It follows that
	$$t(A,f) \leq t(A,f_0|_V) + \frac{|V(A)|}{|V|}\|f-f_0|_V \|_1 \stackrel{\eqref{e:bdv}}{\leq} \delta,$$ 
	and so $f$ is as required.
\end{proof}

\section{Critical number and $k$-affine patterns}\label{sec:kaffine}

The main result of this section, \cref{c:kAffine} shows that the critical number of a hereditary property of binary matroids can be alternatively defined as the maximum $k$ such that  for every  $k$-affine pattern $B$ some  matroid in $\mc{P}$ contains a $B$-instance. For $k=1$ this implies the equivalence of \cref{t:thin} (a) and (b).

First, we need a packing result for rooted subspaces of a binary projective space.

\begin{lem}\label{l:extensions}
	Let $U \subset W \subset V$ be binary projective spaces. Let $d=\dim(V)-\dim(W)+\dim(U)$. Then there exist subspaces $U_1,U_2,\ldots,U_{m}$ of $V$ such that \begin{itemize} \item $m = 2^{\dim(V)-2d}$, \item  $\dim(U_i)=d$ and $U_i \cap W = U$  for every $i \in [m]$,  and \item $U_i \cap U_j = U$ for every $\{i,j\} \subseteq [m]$.
		\end{itemize} 
\end{lem}

\begin{proof} Let $d'=\dim(U)$, $d''=\dim(W)$, then $\dim(V)= d'' + d-d'$.
	Let $U_1,U_2,\ldots,U_{m}$ be a maximal collection of subspaces satisfying the last two conditions of the lemma. Let a subspace $U'$ of $V$ satisfying $\dim(U')=d$ and $U' \cap W = U$ be chosen uniformly at random. Note that for every pair $v_1,v_2 \in V - W$ there exists an automorphism  of $V$ that acts trivially on $W$ and maps $v_1$ to $v_2$. Thus $$\Pr[v \in U'] = \frac{2^{d} - 2^{d'}}{2^{d'' +d-d'} - 2^{d''}}={2}^{d' -d''}$$
	for every $v \in V - W$.
	It follows that
	$$\e[|(U' \cap \cup_{i \in [k]} U_i)-W|]= m(2^{d} - 2^{d'}){2}^{d' -d''} \leq m \cdot2^{d'+d - d''}$$
	If $m< 2^{d'' -d'-d}$ then the above expectation is less than one. Thus in this case there exists a choice
	of $U'$ such that $U' \cap U_i = U$ for every $i \in [k]$, and the collection  $U_1,U_2,\ldots,U_{m}, U'$ contradicts maximality of $U_1,U_2,\ldots,U_{m}$. It follows that $m \geq 2^{d'' -d'-d} =  2^{\dim(V)-2d}$, as desired.
\end{proof}

\begin{lem}\label{c:countExt} For every $d \in \bb{N}$ there exists $\eps=\eps_{\ref{c:countExt}}(d)$ satisfying the following. Let $n \geq d \geq k \geq 1$ be integers. Let $V$ be a binary projective space with $\dim(V)=n$, let $W \in \Sub(k,V)$, and let $N$ be a pattern with $\dim(N)=d-k$. Let $M$   be a matroid with $V(M)=W$ containing an $N$-instance. Then for every $N' \in \brm{Ext}^k(N)$ there are at most $$2^{2^n(1 -2^{-k}-\eps)}$$  $N'$-free extensions of $M$ to $V$.
\end{lem}

\begin{proof} We assume $\dim(N')=d$ without loss of generality.
	Let $U$ be a subspace of $W$ such that $M[U]$ is isomorphic to $N$, 
	let  $U_1,U_2,\ldots,U_{m}$ be the subspaces of $V$ satisfying the conditions of \cref{l:extensions}.

	As at least one extension of $M[U]$ to $U_i$ is isomorphic to $N'$, there are at most $2^{2^{d}-2^{d-k}}-1$ possible restrictions  of an $N'$-free extensions of $M$ to $U_i$ for every $i \in [m]$.
 	Let $U_0=(V \setminus  W) \setminus \cup_{i \in [m]}U_i $ then the total number of $N'$-free extensions of $M$ to $V$ is upper bounded by 
 	\begin{align*}2^{|U_0|}&(2^{2^{d}-2^{d-k}}-1)^{m} \leq 2^{|U_0|}2^{m(2^{d}-2^{d-k})}\s{1-2^{-2^{d}}}^m  \\ &\leq 2^{2^n-2^{n-k} - m \cdot 2^{-2^{d}}} \leq 
 	2^{2^n-2^{n-k} - 2^{n-2d-2^{d}}}\\ &\leq  2^{2^n-2^{n-k} -  2^{n-2^{d+1}}},	 \end{align*}
 	implying that the lemma holds with $\eps = 2^{-2^{d+1}}$.
\end{proof}

\begin{cor}\label{l:Bpattern}
	For every finite collection of $k$-affine patterns $\mc{A}$ and every integer $n_0$ there exists a $k$-affine pattern $B$ with  $\dim(B) \geq n_0$ such that for every $A \in \mc{A}$, every linear injection $\phi_{\star}: A^{-1}(\star) \to B^{-1}(\star)$ extends to an $A$-instance $\phi$ in $B$. %such that
	%$\phi^{-1}(B^{-1}(\star)) = A^{-1}(\star)$.
\end{cor}	

\begin{proof}
	Let $V$ be a binary projective space with $\dim(V)=n \geq n_0$, let $W \subset V$ be a  subspace of codimension $k$. Let a random   $k$-affine pattern $B$ with $V=V(B)$ be defined by setting $B(x)=\star$ for every $x \in W$ and by choosing $B(x) \in \{0,1\}$ uniformly and independently at random for every  for every $x \in V \setminus  W$. Thus any particular $B$ is selected with probability $$2^{-|V \setminus W|}=2^{-2^n(1 -2^{-k})}.$$
	
	We claim that if $n$ is sufficiently large as a function of $\mc{A}$ then $B$ satisfies the lemma with positive probability. Indeed,  assume without loss of generality that $\dim(A)=d$ for every $A \in \mc{A}$. Then by \cref{c:countExt} the probability that an injection $\phi_{\star}$ as in our lemma statement does not extend to the desired instance in $B$ is at most $2^{-\eps  2^{n}}$
for some $\eps >0$ independent on $n$.
	As there are at most $|\mc{A}| \cdot 2^{nd}$ possible maps $\phi_{\star}$ to consider,  our claim holds by the union bound, as long as we have
	$$|\mc{A}| \cdot 2^{nd -\eps 2^{n}} < 1,$$ 
	for sufficiently large $n$, which we clearly do.
\end{proof}

A matroid $N : V \to \{0,1\}$ is an \emph{evaluation} of a pattern $B$ if $V(B)=V(N)$ and $B^{-1}(i) \subseteq N^{-1}(i)$ for $i \in \{0,1\}$. Thus an evaluation of $B$ replaces its $\star$ values by zeroes and ones.

\begin{lem}\label{l:merge} For all $N_0 \in \mc{M}(k,0)$ and $N_1 \in \mc{M}(k,1)$ there exists a $k$-affine pattern $B$ such that any evaluation of $B$ contains an $N_i$-instance for some $i \in \{0,1\}$. 
\end{lem}	

\begin{proof}
	Assume without loss of generality that $\dim(N_0)=\dim(N_1)=d$.
	For $i \in \{0,1\}$ let $W_i \in \Sub(V(N_i),k)$ be such that $N_i[W_i]\equiv i$, and let $A_i$ be the $k$-affine pattern obtained from $N_i$ by setting $A(x)=\star$ for every $x \in W_i$. 
	Let $n_0 =n_{\ref{t:MRamsey}}(d-k)+k$, and let $B$ be a $k$-affine pattern that satisfies the conclusion of \cref{l:Bpattern} for $n_0$ and $\mc{A}=\{A_1,A_2\}$. 
	
	We claim that $B$ satisfies the lemma. 
	Let $W = B^{-1}(\star)$ and let $M$ be an evaluation of $B$. As $\dim(W) \geq n_0$ there exists a subspace $U$ of $W$ with $\dim(U)=d-k$ such that $M[U] \equiv i$ for some $i \in \{0,1\}$. By the choice of $B$ there exists
	an $A_i$-instance $\phi: V(A_i) \to V(B)$ such that $\phi(W_i) = U$. It follows that $\phi$ is an $N_i$-instance in $M$, implying the claim. 
\end{proof}

\begin{cor}\label{c:kAffine} Let $k \geq 1$ be an integer, and let $\mc{P}$ be a  hereditary matroid property. Then $\chi_c(\mc{P}) < k$ if and only if there exists a $k$-affine pattern $B$ such that every matroid in $\mc{P}$ is $B$-free.
\end{cor}	
\begin{proof}
	Suppose first that $\chi_c(\mc{P}) \geq k$. Then $\mc{M}(k,i) \subseteq \mc{P}$ for some $i \in \{0,1\}$. In particular, for every $k$-affine pattern $B$, the evaluation of $B$ obtained by replacing every $\star$ value by $i$ is in $\mc{P}$. Thus  for every such $B$ the property $\mc{P}$  contains a matroid which is not $B$-free, implying the ``if'' direction of the corollary.
	
	For the other direction assume that  $\chi_c(\mc{P}) < k$. Then there exist $N_0 \in \mc{M}(k,0)$ and $N_1 \in \mc{M}(k,1)$ such that $N_0,N_1 \not \in \mc{P}$. By \cref{l:merge} there exists  a $k$-affine pattern $B$ such that any evaluation of $B$ contains an $N_i$-instance for some $i \in \{0,1\}$. 
	
	Suppose for a contradiction that for some $M \in \mc{P}$ there exists a $B$-instance in $M$, i.e. a restriction  of $M$ is an evaluation of $B$. By the choice of $B$, it follows that there exists an $N_i$-instance in $M$ for some $i \in \{0,1\}$, contradicting the choice of $N_i$.    This contradiction implies that every matroid in  $\mc{P}$ is $B$-free, as desired.
\end{proof}

\section{Entropy and proofs of the main results}\label{sec:entropy}

Recall that \emph{binary entropy} is a function $h: [0,1] \to [0,1]$ given by $$h(x) = -x\log_2(x) - (1-x)\log_2(1-x)$$ for $x\in(0,1)$, and $h(0) = h(1) = 0$.
 
 The following is well known, see e.g. \cite[Theorem 3.1]{Galvin14}.
 
\begin{proposition}\label{prop:density_bound}
	$$\binom{n}{k} \leq 2^{n\cdot h(k/n)}$$
\end{proposition}

For a function $f \colon X \to [0,1]$, let $H(f) = \sum_{x\in X} h(f(x))$ be the \emph{binary entropy of $f$}. 

\begin{lem}\label{l:num} Let $X$ be a finite set, let $f \colon X \to [0,1]$ then $$|\mc{M}(f)| \leq 2^{H(f)}.$$
\end{lem}

\begin{proof}
For every $f$-structured function  $M :X \to \{0,1\}$ and every $a \in \brm{Im}(f)$ we have $|M^{-1}(1) \cap f^{-1}(a)| =  a |f^{-1}(a)|.$ The number of such functions $M$ is thus upper bounded by
	$$\prod_{a \in \brm{Im}(f)} \binom{|f^{-1}(a)|}{a |f^{-1}(a)|} \leq \prod_{a \in \brm{Im}(f)} 2^{h(a) |f^{-1}(a)|} = 2^{H(f)},$$
	where the inequality holds by \cref{prop:density_bound}.
\end{proof}

Next, in the first of our main technical lemmas we apply stability and removal theorems for -Burton geometries to show that functions $f:V \to [0,1]$, which have low density of some $(k+1)$-affine pattern, but high entropy, admit a low entropy restriction to some $k$-dimensional subspace.

\begin{lem}\label{l:SubEntropy}
	For all $\eps,d > 0$ there exists $\delta=\delta_{\ref{l:SubEntropy}}(\eps,d), D=D_{\ref{l:SubEntropy}}(\eps,d) > 0$ satisfying the following. Let $V$ be a binary projective space with $\dim(V) \geq D$, let $0 \leq k < d$ and let $f:V \to [0,1]$
	such that  $t(A,f) \leq \delta$ for some $(k+1)$-affine pattern $A$ with $\dim(A) \leq d$ and $H(f) \geq (1-1/2^k-\delta)|V|$. Then there exists $W \in \Sub(k,V)$ such that
	$$\log |\, \{M[W] \:|\: M \in \mc{M}(f)\} \,| \leq \eps |V|.$$
\end{lem}

\begin{proof}
Let $\delta_0=\delta_{\ref{t:stability}}(d,\eps/3)$	,  $\eps_1 = \frac{1}{3}\min \{\delta_0,\eps\}$, and $\delta_1 = \delta_{\ref{t:removal}}(d,\eps_1)$. Let $D = D_{\ref{t:stability}}(d,\eps/3).$
Let $\delta_2 > 0$ be chosen so that $h(\delta_2) < \eps_1 $.
Finally, let $\delta$ be chosen so that $$\delta + h(\delta_2)+\eps_1 \leq \delta_0 \qquad \text{and}  \qquad \delta \cdot \delta^{-2^d}_2 \leq \delta_1.$$
We show that $\delta$ and $D$ satisfy the lemma.
	
We assume without loss of generality that $\dim(A) = d$.	 
Let $\hat{A}$ be the pattern obtained from $A$ by changing zero values to one.\footnote{Formally, $\hat{A}$ is the unique pattern with $V(\hat{A})=V(A)$ defined by $\hat{A}(x)=\star$ for $x \in A^{-1}(\star)$, and $\hat{A}(x) = 1$, otherwise.} Note that $\hat{A}$ is isomorphic to $BB_{k+1,d}$.
Let $\hat{f} (x)  = \min \{f(x),1-f(x)\}$ for $x \in V$. Then
\begin{align*}
t(BB_{k+1,d},\hat{f}) &= \mathlarger{\mathop{\e}}_{\phi} \left[\prod_{x \in \hat{A}^{-1}(1)} \hat{f}(\phi(v)) \right] =  \mathlarger{\mathop{\e}}_{\phi} \left[\prod_{x \in A^{-1}(1)} \hat{f}(\phi(v)) \prod_{x \in A^{-1}(0)} \hat{f}(\phi(v)) \right] \\ &\leq   \mathlarger{\mathop{\e}}_{\phi} \left[\prod_{x \in A^{-1}(1)} f(\phi(v)) \prod_{x \in A^{-1}(0)} (1-f(\phi(v))) \right] = t(A,f) \leq \delta.
\end{align*}

Let $$X = \{ x \in X : \hat{f}(x)  \geq \delta_2\},$$  and let  $M$ be a matroid with $V(M)=V$ and $M^{-1}(1)=X$. 
Then $M(x) \leq \delta^{-1}_2\hat{f}(x)$ for every $x \in V$, implying
$$t(BB_{k+1,d},M) \leq \delta^{-2^d}_2\cdot t(BB_{k+1,d},\hat{f}) \leq  \delta^{-2^d}_2 \delta \leq \delta_1.$$
Thus by the choice of $\delta_1$ there exists a $BB_{k+1,d}$-free matroid $M_0$ with $V(M_0)=V$ such that $M_0^{-1}(1) \subseteq X$ and $|M_0^{-1}(1)| \geq |X|  - \eps_1 |V|$.  
Moreover, $$(1-1/2^k-\delta)|V| \leq H(f)=H(\hat{f}) \leq h(\delta_2)|V| + |X|,$$
implying $$|M_0^{-1}(1)| \geq \s{1-\frac{1}{2^k}-\delta-h(\delta_2) -\eps_1}|V| \geq \s{1-\frac{1}{2^k}-\delta_0}|V|.$$
By the choice of $\delta_0$ there exists $W \in \Sub(k,V)$ such that $|M_0^{-1}(1) \cap W| \leq \eps|V|/3$. Thus $|W \cap X|  \leq (\eps/3 + \eps_1)|V|$. 

By \cref{l:num} we have
$$\log |\{M[W - X] | M \in \mc{M}(f)\}| \leq H(f|_{W - X}) \leq h(\delta_2)|V|,$$
implying 
 \begin{align*} \log &|\{M[W] | M \in \mc{M}(f)\}| \leq \log |\{M[W - X] | M \in \mc{M}(f)\}| + \log |\{M[W \cap X] | M \in \mc{M}(f)\}| \\&\leq h(\delta_2)|V|+|W \cap X| \leq   (\eps/3+\eps_1+h(\delta_2))|V| \leq \eps|V|,\end{align*}
as desired.
\end{proof}

We are now ready to derive  \cref{t:thin} from \cref{t:bd},  the $k=1$ case of \cref{c:kAffine}, and the $k=0$ case of \cref{l:SubEntropy}. 

\begin{proof}[Proof of \cref{t:thin}.]
	
	\textbf{(a) $\Rightarrow$ (b):} This application is an immediate consequence of $k=1$ case of \cref{c:kAffine}.
	
	 \textbf{(b) $\Rightarrow$ (c):} Let $A$ be an affine pattern such that every matroid in $\mc{T}$ is $A$-free. Let $\eps > 0$ be arbitrary, let $d=\dim(A)$,  $\delta=\delta_{\ref{l:SubEntropy}}(\eps,d)$, $D=D_{\ref{l:SubEntropy}}(\eps,d)$ and  $C = C_{\ref{t:bd}}(\delta,d)$. 
	 
	 Let $V = PG(n-1,2)$ for $n \geq D$. By the choice of $C$, there exists a set of functions $f_1,f_2,\ldots,f_m: V \to [0,1]$ such that\begin{itemize}
	 	\item $m \leq 2^{n^C}$,
	 	\item $t(A,f_i) \leq \delta$ for every $i \in [m]$,
	 	\item $\mc{T}^n \subseteq \cup_{i \in [m]} \mc{M}(f_i)$. 
	 \end{itemize} 	 
	 By \cref{l:SubEntropy} applied with $k=0$, we have $\log|\mc{M}(f_i)| \leq \eps 2^n$ for every $i \in [m]$. Thus
	 $$h(n,\mc{T}) \leq \log\s{\sum_{i \in [m]}|\mc{M}(f_i)|} \leq  \eps 2^n + \log m  \leq \eps 2^n + n^c.$$
	 As the above inequality holds for every choice of $\eps > 0$ and every $n \geq D_{\ref{l:SubEntropy}}(\eps,d)$, we have $h(n,\mc{T}) = o(2^n)$, as desired.
	 
	\textbf{(c) $\Rightarrow$ (a):}  If (a) does not hold, i.e. $\mc{T}$ is not thin, then $\mc{M}(1,i) \subseteq \mc{T}$ for some $i \in \{0,1\}$. Therefore, $h(n,\mc{T}) \geq 2^{n-1}$ for every $n \geq 1$ and (c) does not hold.
\end{proof}

The next result is our final technical lemma from which the remaining Theorems~\ref{t:main} and~\ref{t:local} readily follow.

\begin{lem}\label{l:general} Let $k \geq 1$ be an integer. Let $\mc{P}$ be a hereditary matroid property with $\chi_c(\mc{P})=k$. Let $\mc{T}$ be a locally characterized matroid  property such that
	$\brm{Core}^k(\mc{P}) \subseteq \mc{T}$. Then there exists $\delta, D > 0$ such that  $$h(n, \mc{P} - \brm{Ext}^k(\mc{T})) \leq \s{1 -\frac{1}{2^k}-\delta}2^n$$
	for all $n \geq D$.
\end{lem}

\begin{proof}
By \cref{c:kAffine} there exists a $(k+1)$-affine pattern $B$ such that every matroid in $\mc{P}$ is $B$-free.	
Let $\mc{N}$ be a finite set of matroids such that $\mc{T}=\brm{Forb}(\mc{N})$. Let $d$ be such that $\dim(B) \leq d$ and  $\dim(N) \leq d-k$ for every $N \in \mc{N}$.  Let $\eps = \frac{1}{2}\eps_{\ref{c:countExt}}(d)$,  $\delta'=\min\{\eps,\delta_{\ref{l:SubEntropy}}(\eps,d)\}, D' =D_{\ref{l:SubEntropy}}(\eps,d)$, and $C= C_{\ref{t:bd}}(\delta',d)$. 

Let $n \geq D'$ be an integer, let $V=PG(n-1,2)$. By the choice of $C$, there exists functions $f_1,f_2,\ldots,f_m: V \to [0,1]$ such that\begin{itemize}
	\item $m \leq 2^{n^C}$,
	\item $t(B,f_i) \leq \delta'$ for every $i \in [m]$,
	\item $\mc{P}^n \subseteq \cup_{i \in [m]} \mc{M}(f_i)$. 
\end{itemize} 	
Let $\mc{R}=\mc{P} - \brm{Ext}^k(\mc{T})$, then \begin{equation}\label{e:hnr} h(n, \mc{R}) \leq \log\s{\sum_{i=1}^m |\mc{R}^n \cap  \mc{M}(f_i)|} \leq n^C + \max_f\log|\mc{R}^n(f)\cap  \mc{M}(f)|,\end{equation}
where the maximum in the above inequality is taken over functions $f: V \to [0,1]$ such that $t(B,f) \leq \delta'$. 
It remains  to upper bound $\log|\mc{R}^n\cap  \mc{M}(f)|$ for such functions $f$.
If  $H(f) \leq  (1-1/2^k-\delta')2^n$  
then
 $$\log|\mc{R}^n(f)\cap  \mc{M}(f)| \leq \log |\mc{M}(f)| \leq  \s{1 -\frac{1}{2^k}-\delta'}2^n$$ by \cref{l:num}.
 
 Assume now that $H(f) \geq  (1-1/2^k-\delta')2^n$. Then  
 by \cref{l:SubEntropy} there exists $W \in \Sub(k,V)$ such that \begin{equation}\label{e:mprime}\log|\mc{M}'| \leq \eps 2^{n},\end{equation}  where $$\mc{M}'= \{M[W] | M \in \mc{R}^n\cap  \mc{M}(f) \}.$$
Fix $M' \in \mc{M}'$ then $M'=M[W]$ for some $M \in \mc{R}^n$. As  $M \not \in \brm{Ext}^k(\mc{T})$ we have $M' \not \in \mc{T}$, thus there exists an $N'$-instance in $M'$ for some $N' \in \mc{N}$. As $N' \not \in \brm{Core}^k(\mc{P})$ there exists a matroid $N  \in \brm{Ext}^k(N')$ such that $N \not \in \mc{P}$. As $\dim N \leq d$ by  the choice of $\eps$ we have
$$|\mc{R}^n \cap \brm{Ext}^k(M')| \leq |\mc{P}^n \cap \brm{Ext}^k(M')| \leq 2^{2^n(1-2^{-k} - 2\eps)}. $$ 
As $\mc{R}^n \cap \mc{M}(f) \subseteq \cup_{M' \in \mc{M}'}(\mc{R}^n \cap \brm{Ext}^k(M'))$, it follows that
 $$\log|\mc{R}^n(f)\cap  \mc{M}(f)| \leq \log |\mc{M}'| + 2^n-2^{n-k} - \eps 2^{n+1} \leq  \s{1 -\frac{1}{2^k}-\delta'}2^n,$$
 i.e. the same bound holds as in the previous case.

By \eqref{e:hnr} we have
$$ h(n, \mc{R})  \leq n^C + \s{1 -\frac{1}{2^k}-\delta'}2^n$$
for $n \geq D'$. Thus the lemma holds for any $\delta < \delta'$ and $D$ sufficiently large as a function of $d$ and $\delta$.
\end{proof}	

\begin{proof}[Proof of \cref{t:main}]
As $\chi_c(\mc{P})=k$ there exist $N_0 \in \mc{M}(k,0), N_1 \in \mc{M}(k,1)$ such that $N_0,N_1 \not \in \mc{P}$. Let $\mc{P}'=\brm{Forb}(\{N_0,N_1\})$ be the set of all matroids which are $N_0$-free and $N_1$-free. Then   $\mc{P} \subseteq \mc{P}'$ and $\chi_c(\mc{P}') \leq k$, implying $\chi_c(\mc{P}')=k$. Thus $\mc{T} = \brm{Core}^k(\mc{P}')$ is non-trivial, and $\mc{T}$ is locally characterized, as $\mc{P}'$ is locally characterized.   
As $h(n ,\mc{P}) \geq \s{1 - \frac{1}{2^k}}2^n$ for $n \geq k$, by \cref{l:general} we have
$$ h(n,\mc{P} - \brm{Ext}^k(\mc{T})) \le h(n ,\mc{P})  - \Omega(2^n),$$
implying (in a very strong sense) that almost every matroid in $\mc{P}$ lies in $ \brm{Ext}^k(\mc{T})$.	
\end{proof}	

\begin{proof}[Proof of \cref{t:local}]
	As in the  proof of \cref{t:main} by \cref{l:general} we have
$$ h(n,\mc{P} - \brm{Ext}^k(\mc{T})) \le h(n ,\mc{P})  - \Omega(2^n),$$
	and so almost every matroid in $\mc{P}$ lies in $ \brm{Ext}^k(\mc{T})$.	
\end{proof}	

\bibliography{snorin}
\bibliographystyle{plain}
\end{document}